\documentclass[10pt]{amsart}
\usepackage{graphicx,color,epsf}
\usepackage [cmtip,arrow]{xy}
\usepackage {pb-diagram,pb-xy} 

\usepackage{amsmath}
\newtheorem{theorem}{Theorem}[section]
\newtheorem{lemma}[theorem]{Lemma}
\newtheorem{corollary}[theorem]{Corollary}
\newtheorem{proposition}[theorem]{Proposition}

\theoremstyle{definition}
\newtheorem{definition}[theorem]{Definition}
\newtheorem{example}[theorem]{Example}

\theoremstyle{remark}
\newtheorem{remark}[theorem]{Remark}

\numberwithin{equation}{section}

\newcommand {\hide}[1]{}
\newtheorem{notation}{Notation}

\newcommand {\junk}[1]{}
\newtheorem{algorithm}{{\bf {\sc Algorithm}}}

\newcommand {\R} {\mbox{\rm R}}
\newcommand {\s}        {\mbox{\rm sign}}

\newcommand {\K}     {\mbox{\rm K}}

\newcommand {\Sphere}{\mbox{${\bf S}$}}     

\newcommand {\Z}  {\mathbb{Z}}
 
\newcommand {\Q}         {\mathbb{Q}}

\newcommand {\RR} {{\mathcal R}}

\newcommand {\la}   {{\langle}}
\newcommand {\ra}   {{\rangle}}
\newcommand {\eps} {{\varepsilon}}
\newcommand {\E} {{\rm Ext}}

\newcommand {\spanof} {{\rm span}}

\def\addots{\mathinner{\mkern1mu
\raise1pt\vbox{\kern7pt\hbox{.}}
\mkern2mu\raise4pt\hbox{.}\mkern2mu
\raise7pt\hbox{.}\mkern1mu}}

\newcommand{\HH}  {\mbox{\rm H}}

\begin{document}
\title[Bounding the Betti numbers]
{Bounding the Betti numbers and computing the Euler-Poincar\'e characteristic
of semi-algebraic sets defined by partly quadratic systems
of polynomials}
\author{Saugata Basu}
\address{School of Mathematics,
Georgia Institute of Technology, Atlanta, GA 30332, U.S.A.}
\email{saugata.basu@math.gatech.edu}
\author{Dmitrii V. Pasechnik}
\address{
MAS Division, School of Physical and Mathematical Sciences,
Nanyang Technological University,
21 Nanyang Link,
Singapore 637371.
}
\email{dima@ntu.edu.sg}
\author{Marie-Fran\c{c}oise Roy}
\address{IRMAR (URA CNRS 305), 
Universit\'{e} de Rennes I,
Campus de Beaulieu 35042 Rennes cedex FRANCE.}
\email{marie-francoise.roy@univ-rennes1.fr}

\thanks{The first author was supported in part by an NSF 
grant CCF-0634907. 
The second author was supported
by an NTU startup grant.
Part of this work was done when the authors
were visiting the Institute of
Mathematics and Its Application, Minneapolis, IRMAR (Rennes) and
the Nanyang Technological University, Singapore.} 

\subjclass{Primary 14P10, 14P25; Secondary 68W30}


\keywords{Betti numbers, Quadratic inequalities, Semi-algebraic sets}

\begin{abstract}
Let $\R$ be a real closed field,
$
{\mathcal Q} \subset \R[Y_1,\ldots,Y_\ell,X_1,\ldots,X_k],
$
with
$
\deg_{Y}(Q) \leq 2, \deg_{X}(Q) \leq d, Q \in {\mathcal Q}, \#({\mathcal Q})=m,$ 
and
$
{\mathcal P}
\subset \R[X_1,\ldots,X_k]
$
with 
$\deg_{X}(P) \leq d, P \in {\mathcal P}, \#({\mathcal P})=s$,
and $S \subset \R^{\ell+k}$  
a semi-algebraic set defined by a 
Boolean formula without negations, with
atoms  $P=0, P \geq 0, P \leq 0, \;
P \in {\mathcal P} \cup {\mathcal Q}$.
We prove that the sum of the Betti numbers of $S$ 
is bounded by
\[
\ell^2 (O(s+\ell+m)\ell d)^{k+2m}. 
\]
This is a common generalization of  previous results in 
\cite{B99} and \cite{Bar97}
on bounding the Betti numbers of closed semi-algebraic sets 
defined by polynomials
of degree $d$ and $2$, respectively.

We also describe an algorithm for computing the Euler-Poincar\'e
characteristic of such sets,
generalizing similar algorithms described in \cite{B99,Bas05-euler}.
The complexity of the algorithm is bounded by
$(\ell  s  m  d)^{O(m(m+k))}$.
\end{abstract}

\maketitle

\section{Introduction and Main Results}
\label{sec:intro}
Let $\R$ be a real closed field and $S \subset \R^k$ a semi-algebraic set
defined by a
Boolean formula with atoms of the form
$P > 0, P < 0, P=0$ for $P \in {\mathcal P} \subset \R[X_1,\ldots,X_k]$.
We call $S$ a ${\mathcal P}$-semi-algebraic set 
and the Boolean formula defining $S$ a ${\mathcal P}$-formula.
If, instead, the Boolean formula
has atoms of the form $P=0, P \geq 0, P \leq 0, \;P \in {\mathcal P}$, 
and additionally contains no negation,
then we will call $S$ a ${\mathcal P}$-closed semi-algebraic set,
and the formula defining $S$ a ${\mathcal P}$-closed formula. 
Moreover, we call a ${\mathcal P}$-closed semi-algebraic set $S$ basic
if the ${\mathcal P}$-closed formula defining $S$ is a conjunction of
atoms of the form $P=0, P \geq 0, P \leq 0, \;P \in {\mathcal P}$.

For any closed semi-algebraic set $X \subset \R^k$, and any field of
coefficients $\K$, we denote by $b_i(X,\K)$ the dimension  of the
$\K$-vector space, $\HH_i(X,\K)$, which is the $i$-th homology group of $X$
with coefficients in $\K$.  We refer to \cite{BPRbook2} for the definition of
homology in the case of $\R$ being 
an arbitrary real closed field, 
not necessarily the field of real numbers, 
and $\K = \Q$.  
The definition for a more general $\K$ is similar.
We denote by $b(X,\K)$ the sum $\sum_{i \geq 0}
b_i(X,\K)$.  We write $b_i(X)$ for $b_i(X,\Z/2\Z)$ and $b(X)$ for
$b(X,\Z/2\Z)$.  Note that the mod-$2$ Betti numbers $b_i(X)$ are an upper bound
on the Betti numbers $b_i(X,\Q)$ (as a consequence of the Universal Coefficient
Theorem for homology (see \cite{Hatcher} for example)).

The following result appeared in \cite{B99}.
\begin{theorem} \cite{B99}
\label{the:B99}
For a ${\mathcal P}$-closed semi-algebraic set $S \subset \R^k$,
$b(S,\K)$ is bounded by $(O(sd))^k$,
where 
$s = \#({\mathcal P})$, and $d = \max_{P \in {\mathcal P}} \deg(P)$. \qed
\end{theorem}

It is a generalization
of the results due to Oleinik, Petrovsky \cite{OP}, Thom \cite{T}, and
Milnor \cite{Milnor2} on bounding the Betti numbers of real varieties. 
It provides an upper bound on the sum of the 
Betti numbers of ${\mathcal P}$-closed semi-algebraic sets in terms of
the number and degrees of the polynomials in ${\mathcal P}$
(see also \cite{BPR02} for a slightly more precise bound, and
\cite{GaV} for an extension of this result to arbitrary semi-algebraic sets
with a slight worsening of the bound).
Notice that this upper bound has singly exponential dependence on $k$, 
and this dependence is unavoidable (see Example \ref{exa:index} below).

In another direction, a restricted class of semi-algebraic sets - namely, 
semi-algebraic sets defined by quadratic inequalities -- has  been considered
by several researchers \cite{Bar93,Bar97, GrPa04, Bas05-first-Kettner}. 
As in the case of general semi-algebraic sets, the Betti numbers 
of such sets can be exponentially large in the number of variables, 
as can be seen in the following example.

\begin{example}
\label{exa:index}
The set~$S\subset\R^\ell$ defined by 
\[
Y_1(Y_1-1)\ge0,\ldots, Y_\ell(Y_\ell-1)\ge0
\]
satisfies $b_0(S)=2^\ell$.
\end{example}

However, it turns out that for a semi-algebraic set $S \subset \R^{\ell}$
defined by $m$ quadratic inequalities,
it is possible to obtain upper bounds on the Betti numbers of $S$ 
which are polynomial in $\ell$ and exponential only in $m$.
The first such result is due to Barvinok \cite{Bar97}, 
who proved the following  theorem.

\begin{theorem}\cite{Bar97}
\label{the:barvinok}
Let $S \subset \R^{\ell}$ be defined by $Q_1 \geq 0, \ldots,  Q_m \geq 0$, 
$\deg(Q_i) \leq 2, 1 \leq i \leq m$. Then
$b(S,\K) \leq \ell^{O(m)}$.
\end{theorem}

A tighter bound appears in \cite{Bas05-first-Kettner}.

Even though Theorem \ref{the:B99}  \cite{B99} and Theorem \ref{the:barvinok} \cite{Bar97} 
are stated and proved in the case $\K=\Q$ in the original papers, 
the proofs can be extended without any 
difficulty
to a general $\K$.
\begin{remark}
Notice that the bound in Theorem \ref{the:barvinok} is polynomial
in the dimension $\ell$ for fixed $m$, and this fact depends crucially on the 
assumption that the degrees of the polynomials $Q_1,\ldots,Q_m$ are 
at most two.
For instance, the semi-algebraic set defined by a {\em single} 
polynomial of degree $4$ can have Betti numbers exponentially large in 
$\ell$, as exhibited  by the 
semi-algebraic subset of $\R^\ell$ defined by
$$
\displaylines{
\sum_{i=0}^{\ell} Y_i^2(Y_i - 1)^2 \leq 0.
}
$$
The above example illustrates  the delicate nature of the bound in Theorem
\ref{the:barvinok}, since a single inequality of degree $4$ is enough to 
destroy the polynomial nature of the bound. In contrast to this, 
we show in this paper (see Theorem \ref{the:main} below)
that a polynomial bound on the Betti numbers of $S$ continues to hold, 
even if we allow a few (meaning any constant number) of the variables 
to occur with degrees larger than two in the polynomials 
 used to describe the set $S$.
\end{remark}

We now state the main results of this paper.

\subsection{Bounds on the Betti Numbers}
We consider semi-algebraic sets defined by polynomial inequalities,
in which the dependence of the 
polynomials on a {\em subset of the variables}  is at most quadratic.
As a result we obtain common generalizations of the bounds stated in 
Theorems \ref{the:B99} and \ref{the:barvinok}.
Given any polynomial $P \in \R[X_1,\ldots,X_k,Y_1,\ldots,Y_\ell]$, we will
denote by $\deg_X(P)$ (resp. $\deg_Y(P)$) the total degree of $P$ with respect
to the variables $X_1,\ldots,X_k$ (resp. $Y_1,\ldots,Y_\ell$).
\begin{notation}
\label{not:PQ}
Throughout the paper we fix a real closed field $\R$,
and denote by
\begin{itemize}
\item
${\mathcal Q}\subset  \R[Y_1,\ldots,Y_\ell,X_1,\ldots,X_k]$,
a family of polynomials
with 
\[
\deg_{Y}(Q) \leq 2, 
\deg_{X}(Q) \leq d,  Q\in {\mathcal Q}, \#({\mathcal Q})=m,
\]
\item
${\mathcal P} \subset \R[X_1,\ldots,X_k]$,
a family of polynomials 
with
\[
\deg_{X}(P) \leq d, P \in {\mathcal P}, \#({\mathcal P})=s.
\]
\end{itemize}
\end{notation}

We prove the following theorem.

\begin{theorem}
\label{the:main}
Let $S \subset \R^{\ell+k}$ 
be a $({\mathcal P} \cup {\mathcal Q})$-closed semi-algebraic set. Then
$$
\displaylines{
b(S) \leq 
\ell^2 (O(s+\ell+m)\ell d)^{k+2m}. 
}
$$
In particular, for $m \leq \ell$, we have
$
\displaystyle{
b(S) \leq \ell^2 (O(s+\ell)\ell d)^{k+2m}. 
}
$
\end{theorem}

Notice that Theorem \ref{the:main} can be seen as a common generalization of 
Theorems \ref{the:B99} and \ref{the:barvinok},
in the sense that 
we recover similar bounds (that is bounds having the same shape) 
as in Theorem \ref{the:B99} (respectively, Theorem \ref{the:barvinok})
by setting $\ell$ and $m$ 
(respectively, $s$, $d$ and $k$) to $O(1)$.
Since we use Theorem \ref{the:B99} in the proof of Theorem \ref{the:main} (more precisely in the proof of 
Theorem \ref{the:homogeneous} which is  a key step in the proof of Theorem \ref{the:main}), 
our proof does not give a new proof of Theorem \ref{the:B99}.
However, our methods do give a 
new proof of the known bound on Betti numbers in 
the quadratic case (Theorem \ref{the:barvinok}),
and this new proof is quite different
from those given in \cite{Bar97,Bas05-first-Kettner,GrPa2}.
The techniques used in \cite{Bar97,GrPa2,Bas05-first-Kettner} do not
appear to generalize easily to the parametrized situation 
considered in this paper.

Note also that as a special case of Theorem \ref{the:main}
we obtain a bound on the sum of the Betti numbers of
a semi-algebraic set defined over a quadratic map.
Such sets have been considered from an algorithmic point of view
in \cite{GrPa04}, where 
an efficient algorithm is described for computing sample 
points in every connected component, as well as testing emptiness, 
of such sets.

More precisely, we show the following.
\begin{corollary}
\label{cor:main}
Let $Q = (Q_1,\ldots,Q_k) : \R^{\ell} \rightarrow \R^k$ be a map where
each $Q_i \in \R[Y_1,\ldots,Y_\ell]$ and $\deg(Q_i) \leq 2$.
Let $V \subset \R^k$ be a ${\mathcal P}$-closed semi-algebraic set for
some family ${\mathcal P} \subset \R[X_1,\ldots,X_k]$, with
$\#({\mathcal P}) = s$ and $\deg(P) \leq d, P \in {\mathcal P}$.
Let $S = Q^{-1}(V)$. Then
$$
\displaylines{
b(S) \leq \ell^2 (O(s+\ell+k)\ell d)^{3k}. 
}
$$
\end{corollary}

\begin{remark}
Note that the Morse theoretic techniques developed in \cite{GrPa2} give 
a possible alternative approach for proving Corollary \ref{cor:main}.
\end{remark}

\subsection{Algorithmic Implications}
The algorithmic problem of computing topological invariants of semi-algebraic
sets (such as the Betti numbers, Euler-Poincar\'e characteristic) is 
very well studied. We refer the reader to a recent survey 
\cite{Basu_survey}
for a detailed account of the recent progress and 
open problems in this field.  

The techniques developed in this paper for obtaining tight bounds on the
Betti numbers of semi-algebraic sets defined by partly quadratic systems of
polynomials also pave the way towards designing more 
efficient algorithms for computing the Euler-Poincar\'e characteristic as
well as the Betti numbers of such sets. These algorithms have better 
complexity than the ones known before.

\begin{definition}[Complexity]
By complexity of an algorithm we will mean the number of arithmetic
operations (including comparisons) performed by the algorithm
in $\R$.
We refer the reader to \cite[Chapter 8]{BPRbook2} for a full discussion
about the various measures of complexity.
\end{definition}

We prove the following theorem.
\begin{theorem}
\label{the:algo-EP}
There exists an algorithm 
that takes as input the description of a
$({\mathcal P} \cup {\mathcal Q})$-closed semi-algebraic set
$S$ (following the same notation as in Theorem \ref{the:main})
and outputs its
the Euler-Poincar\'e characteristic
$\chi(S)$. 
The complexity of this algorithm  is bounded by 
$(\ell s m d)^{O(m(m+k))}$. 
In the case when $S$ is a basic closed semi-algebraic set
the complexity of the algorithm is 
$(\ell s m d)^{O(m+k)}$.
\end{theorem}

The algorithm for computing all the Betti numbers has complexity
$(\ell s m d)^{2^{O(m+k)}}$ and 
is much more technical. We omit its description in this paper.
It will appear in full detail separately in a subsequent paper.
 
While the complexity of both the algorithms discussed above 
is {\em polynomial}  for fixed $m$ and $k$, the complexity of the algorithm
for computing the Euler-Poincar\'e characteristic is significantly better
than that of the algorithm for computing all the Betti numbers.

\subsubsection{Significance from the computational 
complexity theory viewpoint}
The problem of computing the Betti numbers of 
semi-algebraic sets in general is a PSPACE-hard problem.
We refer the reader to \cite{Bas05-top} and the references
contained therein, for a detailed discussion of these hardness results.
In particular, the problem of computing the Betti numbers of a 
real algebraic variety defined by
a single quartic equation is also PSPACE-hard, and the same is true for
semi-algebraic sets defined by many quadratic inequalities. 
On the other hand, as shown in \cite{Bas05-top} (see also \cite{Bas05-top-errata}),
the problem of computing the Betti numbers of semi-algebraic
sets defined by a constant number of quadratic inequalities is solvable in
polynomial time. The results mentioned above indicate that 
the problem of computing the
Betti numbers of semi-algebraic sets defined by a constant number of
polynomial inequalities is solvable in polynomial time, even if we allow
a small (constant sized) subset of the variables to occur with degrees
larger than two in the polynomials defining the given set. 
Note that such a result is not obtainable directly from the results 
in \cite{Bas05-top} by the naive method of 
replacing the monomials having degrees larger than two by a larger set
of quadratic ones (introducing new variables and equations in the process).

For general semi-algebraic sets, the algorithmic problem of computing all
the Betti numbers is notoriously difficult and only doubly exponential
time algorithm is known for this problem. Very recently, singly exponential
time algorithms \cite{BPRbettione,Bas05-first}
have been found for computing the first few Betti numbers
of such sets, but the problem of designing singly exponential time
algorithm for computing all the Betti numbers remains open. Singly exponential
time algorithm is also known for computing the Euler-Poincar\'e characteristic
of general semi-algebraic sets \cite{B99}.

The rest of the paper is organized as follows. In Section \ref{sec:proof}
we prove Theorem  \ref{the:main}. In Section \ref{sec:algo_ep} 
we describe our algorithm
for computing the Euler-Poincar\'e characteristic 
of sets defined by partly quadratic system of polynomials and 
prove Theorem \ref{the:algo-EP}.

\section{Proof of Theorem \ref{the:main}}
\label{sec:proof}
One of the main ideas  behind our proof of Theorem \ref{the:main} is
to parametrize a construction introduced by 
Agrachev in \cite{Agrachev} while studying the topology of sets defined by 
(purely) quadratic inequalities (that is without the parameters 
$X_1,\ldots,X_k$ in  our notation). 
In {\em loc. cit.} Agrachev 
constructs a spectral sequence converging to the cohomology
of the set being studied. 
However, it is assumed that the initial quadratic polynomials are  generic.
In this paper we do not make any genericity assumptions
on our polynomials. In order to prove our main theorem 
we follow 
another approach based on infinitesimal deformations
which avoids the construction of a spectral sequence as done in 
\cite{Agrachev}.

We first need to fix some notation and a few preliminary results
needed later in the proof.

\subsection{Mathematical Preliminaries}
\subsubsection{Some Notation}
For all $a \in R$ we define 
\begin{eqnarray*}
\s(a) &=& 0  \mbox{ if } a = 0, \\
      &=& 1  \mbox{ if } a > 0, \\
      &=& -1 \mbox{ if } a < 0.
\end{eqnarray*}
Let ${\mathcal A}$ be a finite subset of  $\R[X_1,\ldots,X_k]$.
A  {\em sign condition}  on
${\mathcal A}$ is an element of $\{0,1,- 1\}^{\mathcal A}$.
The {\em realization of the sign condition}
$\sigma$, $\RR(\sigma,\R^k)$, is the basic semi-algebraic set
$$
        \{x\in \R^k\;\mid\; \bigwedge_{P\in{\mathcal A}} 
\s({P}(x))=\sigma(P) \}.
$$

A  {\em weak sign condition}  on
${\mathcal A}$ is an element of $\{\{0\},\{0,1\},\{0,-1\}\}^{\mathcal A}$.
The {\em realization of the weak sign condition}
$\rho$, $\RR(\rho,\R^k)$, is the basic semi-algebraic set
$$
        \{x\in \R^k\;\mid\; \bigwedge_{P\in{\mathcal A}} 
\s({P}(x)) \in \rho(P) \}.
$$

We
often abbreviate $\RR(\sigma,\R^k)$ by $\RR(\sigma)$, and we 
denote by ${\rm Sign}({\mathcal A})$ the set
of realizable sign conditions
${\rm Sign}({\mathcal A})=\{\sigma \in \{0,1,- 1\}^{\mathcal A} \;\mid\; \RR(\sigma) \neq \emptyset\}$.

More generally, for any ${\mathcal A} \subset \R[X_1,\ldots,X_k]$ and
a ${\mathcal A}$-formula $\Phi$, we 
denote by 
$\RR(\Phi,\R^k)$, or simply $\RR(\Phi)$,  the
semi-algebraic set defined by $\Phi$ in $\R^k$.

\subsubsection{Use of Infinitesimals}
Later in the paper,
we
extend the ground field $\R$ by infinitesimal
elements.
We denote by $\R\langle \zeta\rangle$  the real closed field of algebraic
Puiseux series in $\zeta$ with coefficients in $\R$ (see \cite{BPRbook2} for
more details). 
The sign of a Puiseux series in $\R\langle \zeta\rangle$
agrees with the sign of the coefficient
of the lowest degree term in
$\zeta$. 
This induces a unique order on $\R\langle \zeta\rangle$ which
makes $\zeta$ infinitesimal: $\zeta$ is positive and smaller than
any positive element of $\R$.
When $a \in \R\la \zeta \ra$ is bounded 
from above and below by some elements of $\R$,
$\lim_\zeta(a)$ is the constant term of $a$, obtained by
substituting 0 for $\zeta$ in $a$.
We denote by 
$\R\langle\zeta_1,\ldots,\zeta_n\rangle$ the
field  $\R\langle \zeta_1\rangle \cdots \langle\zeta_n\rangle$
and in this case 
$\zeta_1$ is positive and infinitesimally small compared to $1$, 
and for $1 \leq i \leq n-1$, 
$\zeta_{i+1}$ is positive and infinitesimally small 
compared to $\zeta_i$, which we abbreviate by writing
$0 < \zeta_n \ll \cdots \ll \zeta_1 \ll 1$.

Let $\R'$ be a real closed field containing $\R$.
Given a semi-algebraic set
$S$ in ${\R}^k$, the {\em extension}
of $S$ to $\R'$, denoted $\E(S,\R'),$ is
the semi-algebraic subset of ${ \R'}^k$ defined by the same
quantifier free formula that defines $S$.
The set $\E(S,\R')$ is well defined (i.e. it only depends on the set
$S$ and not on the quantifier free formula chosen to describe it).
This is an easy consequence of the transfer principle (see for instance
\cite{BPRbook2}).

We will need a few results from algebraic topology, which we state here
without proofs, referring the reader to papers where the proofs appear.

\subsubsection{Mayer-Vietoris Inequalities}
The following inequalities are consequences of the Mayer-Vietoris exact
sequence.

\begin{proposition}[Mayer-Vietoris inequalities]
\label{pro:MV}
Let the subsets $W_1, \ldots , W_t \subset \R^n$ be all 
closed.
Then for each $i \geq 0$ we have
\begin{align}\label{eqn:MV1}
{b}_i \left( \bigcup_{1 \le j \le t} W_j \right) \le & \sum_{J \subset \{1, \ldots ,\ t \}}
{b}_{i- \# (J) +1} \left( \bigcap_{j \in J} W_j \right),\\
\label{eqn:MV2}
{b}_i \left( \bigcap_{1 \le j \le t} W_j \right) \le & \sum_{J \subset \{1, \ldots ,\ t \}}
{b}_{i + \# (J) -1} \left( \bigcup_{j \in J} W_j \right).
\end{align}
\end{proposition}

\begin{proof}
See 
for instance 
\cite{BPRbook2}.
\end{proof}

\subsubsection{Topology of Sphere Bundles}

Given a closed and bounded semi-algebraic set $B$, 
a semi-algebraic
$\ell$-sphere bundle over $B$ is given by a continuous 
semi-algebraic map $\pi: E \rightarrow B$, such that for each $b \in B$,
$\pi^{-1}(b)$ is homeomorphic to $\Sphere^\ell$ (the 
$\ell$-dimensional unit sphere in $\R^{\ell +1}$).

We need the following proposition that relates the Betti numbers of $B$
with that  of $E$.
  
\begin{proposition}
\label{pro:sphere_bundle}
Let $B \subset \R^k$ be a closed and bounded semi-algebraic set and 
let $\pi: E \rightarrow B$ be a 
semi-algebraic $\ell$-sphere bundle with base $B$. Then

\begin{equation} 
\label{eqn:sphere_bundle}
b(E) \leq 2\cdot b(B).
\end{equation}
\end{proposition}

\begin{proof}
In case $\ell > 0$, the proposition follows from 
the inequality \eqref{eqn:gysin}
proved in \cite[page~252~(4.1)]{chern-spanier} 
\begin{equation}
\label{eqn:gysin}
P_E(t) \leq P_{\Sphere^{\ell}}(t) P_B(t),
\end{equation}
where 
$P_X(t) = \sum_{i\geq 0} b_i(X) t^i$ denotes  the Poincar\'e polynomial 
of a topological space $X$,
and the inequality holds coefficient-wise.
The inequality \eqref{eqn:sphere_bundle}
holds for the 
Betti numbers with coefficients in $\Q$, as well.

For $\ell = 0$, inequality \eqref{eqn:gysin}
is no longer true for the ordinary Betti numbers,
as can be observed from the example of the two-dimensional torus, which is
a double cover of the Klein bottle. But inequality \eqref{eqn:sphere_bundle}
holds for Betti numbers with $\Z/2\Z$-coefficients. 
This follows from the Leray-Serre spectral sequence of the projection map
$\pi$, 
since the homology with coefficients in a local system in this
case are the same as ordinary homology
(an elementary proof is given in \cite{Bas05-first-Kettner}).  
\end{proof}

We now return to the proof of Theorem \ref{the:main}.

\subsection{Homogeneous Case}
\label{subsec:homogeneous}

\begin{notation}
\label{not:AhWh}
We denote by
\begin{itemize}
\item  ${\mathcal Q}^h$
the  family of polynomials  
obtained by homogenizing ${\mathcal Q}$ 
with respect to the variables $Y$, i.e.
\[
{\mathcal Q}^h = \{Q^h \;\mid\; Q \in {\mathcal Q}\} 
\subset  \R[Y_0,\ldots,Y_\ell,X_1,\ldots,X_k],
\]
where $Q^h=Y_0^2 Q(Y_1/Y_0,\ldots,Y_\ell/Y_0,X_1,\ldots,X_k)$,
\item $\Phi$  a formula defining a ${\mathcal P}$-closed semi-algebraic set $V$,
\item $A^h$ the semi-algebraic set
\begin{equation}
\label{eqn:defofAh}
A^h = \bigcup_{Q \in {\mathcal Q}^h}
\{ (y,x) \;\mid\; |y|=1\; \wedge\; Q(y,x) \leq 0\; \wedge \; \Phi(x)\},
\end{equation}
\item $W^h$ the semi-algebraic set 
\begin{equation}
\label{eqn:defofWh}
W^h = \bigcap_{Q \in {\mathcal Q}^h}
\{ (y,x) \;\mid\; |y|=1\; \wedge\; Q(y,x) \leq 0\; \wedge \; \Phi(x)\}.
\end{equation}
\end{itemize}
\end{notation}

We are going to prove
\begin{theorem}
\label{the:homogeneous}
\begin{eqnarray}
b(A^h) &\leq &
\ell^2 (O((s+\ell+m)\ell d))^{m+k}.
\end{eqnarray}
\end{theorem}

and

\begin{theorem}
\label{the:homogeneous2}
\begin{eqnarray}
b(W^h) &\leq &  \ell^2 (O((s+\ell+m)\ell d))^{m+k}.
\end{eqnarray}
\end{theorem}

Before proving Theorem \ref{the:homogeneous} and  
Theorem \ref{the:homogeneous2} we need a few preliminary results.

Let
\begin{equation}
\label{def:omega}
\Omega = \{\omega \in \R^{m} \mid  |\omega| = 1, \omega_i \leq 0, 1 \leq i \leq m\}.
\end{equation}

Let ${\mathcal Q}=\{Q_1,\ldots, Q_m \}$
and ${\mathcal Q}^h=\{Q_1^h,\ldots, Q_m^h \}$.
For $\omega \in \Omega$ we denote by 
$\la \omega ,{\mathcal Q}^h \ra \in \R[Y_0,\ldots,Y_\ell,X_1,\ldots,X_k]$ the polynomial
defined by 
\begin{equation}
\label{def:omegaq}
\la \omega , {\mathcal Q}^h \ra = \sum_{i=1}^{m} \omega_i Q_i^h.
\end{equation}

For $(\omega,x) \in \Omega \times V$, we 
denote by
$\la \omega , {\mathcal Q}^h \ra (\cdot,x)$ the quadratic form in $Y_0,\ldots,Y_\ell$ 
obtained from $\la \omega , {\mathcal Q}^h \ra$ by specializing $X_i = x_i, 1 \leq i \leq k$.

Let $B \subset \Omega \times \Sphere^{\ell} \times V$ 
be the semi-algebraic set defined by
\begin{equation}
\label{def:B}
B = \{ (\omega,y,x)\mid \omega \in \Omega, y\in \Sphere^{\ell}, x \in 
V,  \; \la \omega, {\mathcal Q}^h \ra (y,x) \geq 0\}.
\end{equation}

We denote by $\varphi_1: B \rightarrow F$ and 
$\varphi_2: B \rightarrow \Sphere^{\ell} \times V$ the two projection maps
(see diagram below).

\[
\begin{diagram}
\node{}
\node{B} \arrow{sw,t}{\varphi_{1}}\arrow[2]{s}\arrow{se,t}{\varphi_{2}} \\
\node{F = \Omega \times V} \arrow{se} \node{} \node{\Sphere^{\ell} \times V} \arrow{sw} \\
\node{}\node{V}
\end{diagram}
\]

The following key proposition was proved by Agrachev \cite{Agrachev}
in the unparametrized
situation, but as we see below it works in the parametrized case as well.
Note that the proposition is quite general 
and does not require quadratic dependence on the variables $Y$
(i.e. the polynomials $Q_i$ need not be quadratic in $Y$).

\begin{proposition}
\label{pro:homotopy2}
The semi-algebraic set $B$ is homotopy equivalent to
$A^h$.
\end{proposition}

\begin{proof}
We first prove that $\varphi_2(B) = A^h.$
If $(y,x) \in A^h,$ 
then there exists some $i, 1 \leq i \leq m,$ such that
$(Q_i^h(y,x) \leq 0) \wedge \Phi(x)$. 
Then for $\omega = (-\delta_{1,i},\ldots,-\delta_{m,i})$
(where $\delta_{ij} = 1$ if $i=j$, and $0$ otherwise),
we see that 
$(\omega,y,x) \in B$.
Conversely,
if $(y,x) \in \varphi_2(B),$ then there exists 
$\omega 
\in \Omega$ such that 
$\la \omega, {\mathcal Q}^h \ra(y,x)
\geq 0$. Since 
$\omega \leq 0$ and $\omega \neq 0$, we have 
that 
$(Q_i^h(y,x) \leq 0) \wedge \Phi(x)$ for
some $i, 1 \leq i \leq m$. This shows that $(y,x) \in A^h$.

For $(y,x) \in \varphi_2(B)$, the fibre 
$$
\varphi_2^{-1}(y,x) = \{ (\omega,y,x) \mid  
 \omega \in \Omega \;\mbox{such that} \;  \la \omega , {\mathcal Q}^h \ra (y,x) \geq 0\},
$$
is a non-empty subset of $\Omega$ defined by a single linear inequality.
Thus, each 
fiber is an intersection of a 
non-empty closed
convex cone with
$\Sphere^{m-1}$.
The proposition now follows from the well-known Vietoris-Smale theorem 
\cite{Smale}
since by the above observation each fiber is a closed, bounded and 
contractible semi-algebraic set.
\end{proof}

We will use the following notation.
\begin{notation}
For a  quadratic form $Q \in \R[Y_0,\ldots,Y_\ell]$, 
we
denote by ${\rm index}(Q)$ the number of
negative eigenvalues of the symmetric matrix of the corresponding bilinear
form, i.e. of the matrix $M$ such that
$Q(y) = \langle M y, y \rangle$ for all $y \in \R^{\ell+1}$ 
(here $\langle\cdot,\cdot\rangle$ denotes the usual inner product). 
We also
denote by $\lambda_i(Q), 0 \leq i \leq \ell$ the eigenvalues of $Q$ 
in non-decreasing order, i.e.
\[ \lambda_0(Q) \leq \lambda_1(Q) \leq \cdots \leq \lambda_\ell(Q).
\]
\end{notation}

For $F=\Omega \times V$ as above we denote
\[
F_j = \{(\omega,x) \in F\;  
\mid \;  {\rm index}(\la \omega , {\mathcal Q}^h \ra (\cdot,x)) \leq j \}.
\]

It is clear that each 
$F_j$ is a closed semi-algebraic subset of 
$F$ and 
we get
a filtration of the space
$F$ 
given by
\[
F_0 \subset F_1 \subset \cdots \subset 
F_{\ell+1} = F.
\]

\begin{lemma}
\label{lem:sphere}
The fibre of the map $\varphi_1$ over a point 
$(\omega,x)\in F_{j}\setminus F_{j-1}$ 
has the homotopy type of a sphere of dimension $\ell-j$. 
\end{lemma}

\begin{proof}
Denote $\lambda_i(\omega,x)=\lambda_i(\la \omega , {\mathcal Q}^h \ra(\cdot,x))$ the
eigenvalues of $\la \omega,  {\mathcal Q}^h \ra (\cdot,x)$ in increasing order.
First notice that for
$(\omega,x) \in  F_{j}\setminus F_{j-1}$,
\[
\lambda_0(\omega , x)\le \cdots \le \lambda_{j-1}( \omega ,x)<0.
\] 
Moreover, letting 
$W_0(\la \omega , {\mathcal Q}^h \ra(\cdot,x)),\ldots,W_{\ell}(\la \omega , {\mathcal Q}^h \ra (\cdot,x))$ 
be the co-ordinates with respect to an orthonormal basis consisting of
eigenvectors of $\la \omega , {\mathcal Q}^h\ra (\cdot,x)$, we have that 
$\varphi_1^{-1}(\omega,x)$ is the subset of 
$\Sphere^{\ell} = \{\omega\} \times \Sphere^{\ell} \times \{x\}$ 
defined by
$$
\displaylines{
\sum_{i=0}^{\ell} \lambda_i(\omega, x)W_i(\la \omega ,{\mathcal Q}^h \ra (\cdot,x))^2 \geq  0, \cr
\sum_{i=0}^{\ell} W_i(\la \omega ,{\mathcal Q}^h\ra (\cdot,x))^2 = 1.
}
$$

Since $\lambda_i(\omega ,x) < 0$  for all 
$0 \leq i < j,$ it follows that
for $(\omega,x) \in F_{j}\setminus 
F_{j-1}$,
the fiber $\varphi_1^{-1}(\omega,x)$ is homotopy equivalent to the
$(k-j)$-dimensional sphere defined by setting
\[
W_0(\la \omega, {\mathcal Q}^h \ra (\cdot,x)) = \cdots = W_{j-1}(\la \omega ,{\mathcal Q}^h \ra (\cdot,x)) = 0
\]
on the sphere defined by
\[
\sum_{i=0}^{\ell}W_i(\la \omega , {\mathcal Q}^h \ra (\cdot,x))^2 = 1.
\]
\end{proof}

For each 
$(\omega,x) \in F_j \setminus F_{j-1}$, let 
$L_j^+(\omega,x) \subset \R^{\ell+1}$ denote the sum of the
non-negative eigenspaces of 
$\la \omega , {\mathcal Q}^h \ra (\cdot,x)$.
Since  ${\rm index}(\la \omega, {\mathcal Q}^h \ra (\cdot,x)) = j$ stays invariant as
$(\omega,x)$ varies over $F_j
\setminus F_{j-1}$,
$L_j^+(\omega,x)$ varies continuously with $(\omega,x)$.

We denote by $C$ the semi-algebraic set defined by the following.
We first define for $0 \leq j \leq \ell+1$
\begin{equation}
\label{eqn:definition_of_C_j}
C_j = \{(\omega,y,x) \;\mid\; (\omega,x) \in 
      F_{j}\setminus F_{j-1}, 
y \in L_j^+(\omega,x), |y| = 1\},
\end{equation}
and finally we define
\begin{equation}
\label{eqn:definition_of_C}
C = \bigcup_{j=0}^{\ell+1} C_j.
\end{equation}

The following proposition relates the homotopy type of $B$ to that
of $C$.

\begin{proposition}
\label{pro:homotopy1}
The semi-algebraic set $C$ defined by 
\eqref{eqn:definition_of_C} is homotopy equivalent to $B$.
\end{proposition}

Before proving the Proposition we give an illustrative example.

\begin{example}
In this example $m=2,\ell = 3,k=0$, and 
${\mathcal Q}^h=\{Q_1^h,Q_2^h\}$ with
\begin{align*}
Q_1^h =& - Y_0^2 - Y_1^2 - Y_2^2, \\
Q_2^h =&   Y_0^2 + 2 Y_1^2  + 3 Y_2^2.
\end{align*}

The set $\Omega$ is the part of the unit circle in the third quadrant of the
plane, 
and  $F = \Omega$ in this case (since $k=0$).
In the following Figure \ref{fig:illus}, we display
the fibers of the map $\varphi_1^{-1}(\omega) \subset B$ for a sequence of 
values of $\omega$ starting from $(-1,0)$ and ending at 
$(0,-1)$. We also show the spheres,
$C \cap \varphi_1^{-1}(\omega)$, of dimensions $0,1$, and $2$, that these fibers
retract to. At $\omega = (-1,0)$, it is easy to verify that
${\rm index}(\la \omega, {\mathcal Q}^h \ra) = 3$, and the 
fiber $\varphi_1^{-1}(\omega) \subset B$
is empty. Starting from
$\omega = (-\cos(\arctan(1)),-\sin(\arctan(1)))$ we have 
${\rm index}(\la \omega ,{\mathcal Q}^h\ra) = 2$,
and the fiber 
$\varphi_1^{-1}(\omega)$ consists of the union of two spherical caps, 
homotopy equivalent to $\Sphere^0$.
Starting from
$\omega = (-\cos(\arctan(1/2)),-\sin(\arctan(1/2)))$ we have
${\rm index}(\la \omega,  {\mathcal Q}^h\ra) = 1$, and 
the fiber $\varphi_1^{-1}(\omega)$ is homotopy equivalent to $\Sphere^1$. Finally,
starting from
$\omega = (-\cos(\arctan(1/3)),-\sin(\arctan(1/3)))$,
${\rm index}(\la \omega , {\mathcal Q}^h\ra) = 0$, and 
the fiber $\varphi_1^{-1}(\omega)$ stays equal to to $\Sphere^2$.

\begin{figure}[hbt]
\scalebox{0.8}{
\begin{picture}(405,70)
\includegraphics{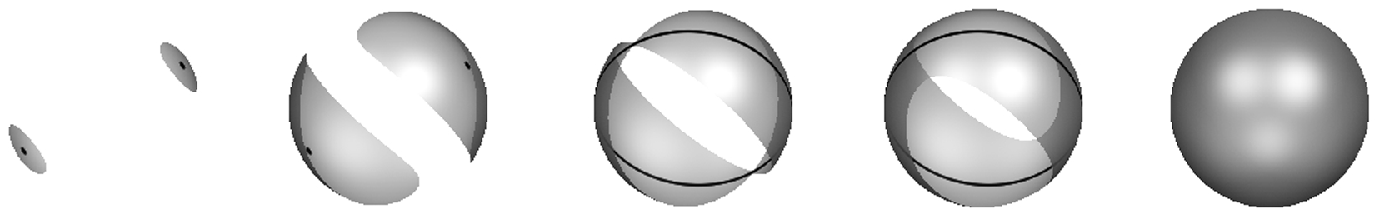}%
\end{picture}
}
\caption{Type change: $\emptyset\to \Sphere^0\to \Sphere^1\to \Sphere^2$. 
$\emptyset$ is not shown. }
\label{fig:illus}
\end{figure}
\end{example}

\begin{proof}[Proof of Proposition \ref{pro:homotopy1}]
We construct a deformation retraction of $B$ to $C$ as follows.  Let 
\begin{equation}
\label{eqn:definition_of_B_j}
B_j = \bigcup_{i=j}^{\ell+1} C_i \cup \varphi_1^{-1}(F_{j-1}),
\end{equation}
and note that $B_{\ell+1}=B,\ldots,B_0=C$.

We construct a sequence of homotopy equivalences from $B_{j+1}$ to $B_{j}$ 
for every $j=\ell,\ldots,0$ as follows.

Let $0\le j \le \ell$. For each $(\omega,x) \in F_{j} \setminus 
F_{j-1}$, we  retract the fiber 
$\varphi_1^{-1}(\omega,x)$ to the $(\ell-j)$-dimensional sphere,
$L_{j}^+(\omega,x) \cap \Sphere^{\ell}$
as follows.
Let
\[
W_0(\la \omega ,{\mathcal Q}^h \ra (\cdot,x)),\ldots,
W_{\ell}(\la \omega ,{\mathcal Q}^h \ra (\cdot,x))
\]
be the co-ordinates with respect to an orthonormal basis consisting of
eigenvectors
\[
e_0(\la \omega ,{\mathcal Q}^h \ra (\cdot,x)),\ldots,e_{\ell}(\la \omega ,{\mathcal Q}^h \ra (\cdot,x))
\]
of $\la \omega , {\mathcal Q}^h \ra(\cdot,x)$ corresponding to 
the non-decreasing sequence of eigenvalues 
of $\la \omega , {\mathcal Q}^h \ra(\cdot,x)$. 
Then
$\varphi_1^{-1}(\omega,x)$ is the subset of 
$\Sphere^{\ell}$ 
defined by
$$
\displaylines{
\sum_{i=0}^{\ell} \lambda_i(\omega ,x)W_i(\la \omega ,{\mathcal Q}^h \ra(\cdot,x))^2 \geq  0, \cr
\sum_{i=0}^{\ell} W_i(\la \omega ,{\mathcal Q}^h \ra(\cdot,x))^2 = 1.
}
$$
and $L_{j}^+(\omega,x)$ is defined by $W_0(\la \omega ,{\mathcal Q}^h \ra (\cdot,x)) = \cdots = 
W_{j-1}(\la \omega ,{\mathcal Q}^h \ra (\cdot,x)) = 0$.
We retract 
$\varphi_1^{-1}(\omega,x)$ to the $(\ell-j)$-dimensional sphere,
$L_{j}^+(\omega,x) \cap \Sphere^{\ell}$
by the retraction sending,
$(w_0,\ldots,w_\ell) \in \varphi_1^{-1}(\omega,x)$, at time $t$ to
$(tw_0,\ldots,tw_{j-1},t'w_j,\ldots,t'w_\ell)$, where $0 \leq t \leq 1$,
and 
$
\displaystyle{
t' = \left(\frac{1 - t^2 \sum_{i=0}^{j-1}w_i^2}{\sum_{i=j}^{\ell}
w_i^2}\right)^{1/2}.
}
$
Notice that even though the local co-ordinates 
$W_0(\la \omega ,{\mathcal Q}^h \ra (\cdot,x)),\ldots,W_{\ell}(\la \omega {\mathcal Q}^h \ra(\cdot,x))$ 
in $\R^{\ell+1}$ 
with respect to the orthonormal basis
$(e_0(\la \omega, {\mathcal Q}^h\ra (\cdot,x)),\ldots,e_{\ell}(\la \omega ,{\mathcal Q}^h\ra (\cdot,x)))$ 
of eigenvectors
may not be uniquely defined at the point $(\omega,x)$ 
(for instance,
if the quadratic form $\la \omega,{\mathcal Q}^h\ra(\cdot,x)$ has multiple eigen-values),
the retraction is still well-defined since it only depends on the
decomposition of 
$\R^{\ell+1}$ into 
orthogonal complements $\spanof(e_0,\ldots,e_{j-1})$ and 
$\spanof(e_j,\ldots,e_\ell)$ which is well defined. 

We can thus retract simultaneously all fibers
over $F_{j} \setminus F_{j-1}$
continuously, to obtain  $B_{j} \subset B$,
which is moreover homotopy equivalent to $B_{j+1}$.
\end{proof}

Notice that the semi-algebraic set 
$C_j$
is a
$\Sphere^{\ell - j}$-bundle over $F_j \setminus F_{j-1}$ under the
map $\varphi_1$, and $C$ is a union of these sphere bundles.
 Since we have good bounds 
on the number as well as the degrees of polynomials 
used to define the bases, $F_j \setminus F_{j-1}$, 
we can bound the Betti numbers of each 
$C_j$
using 
Proposition \ref{pro:sphere_bundle}. However, 
the $C_j$
could be possibly 
glued to each other in complicated ways, and thus knowing upper bounds on the
Betti numbers of each 
$C_j$
does not immediately produce a bound on 
Betti numbers of $C$. In order to get around this difficulty, we consider 
certain closed subsets, $F_j'$ of $F$, where each $F_j'$ is an infinitesimal
deformation of $F_j \setminus F_{j-1}$, 
and form the base of a $\Sphere^{\ell - j}$-bundle 
$C'_j$. 
Additionally, 
the $C'_j$
are glued to each other along sphere bundles over
$F_j' \cap F_{j-1}'$, and their union, $C'$,  is homotopy equivalent to $C$. 
Since 
the $C'_j$
are closed and bounded semi-algebraic sets, and we
have good bounds on their Betti numbers as well as the Betti numbers 
of their non-empty intersections, we can use Mayer-Vietoris inequalities
(Proposition \ref{pro:MV}) to bound the Betti numbers of $C'$, which in turn
are equal to the Betti numbers of $C$.

We now make precise the argument outlined above.

Let
$\Lambda \in \R[Z_1,\ldots,Z_m,X_1,\ldots,X_k,T]$ be the
polynomial defined by

\begin{eqnarray*}
\Lambda  &=& \det( T \cdot {{\rm Id}}_{\ell+1}-M_{Z \cdot {\mathcal Q}^h}),\\
   &=&  T^{\ell+1} + C_{\ell} T^\ell + \cdots + C_0,
\end{eqnarray*}
where $Z \cdot {\mathcal Q}^h = \sum_{i=1}^m Z_i Q_i^h$, and 
each $C_i \in \R[Z_1,\ldots,Z_m,X_1,\ldots,X_k]$.

Note that for $(\omega,x) \in \Omega \times\R^k$, the polynomial
$\Lambda(\omega,x,T)$, 
being the characteristic polynomial of a real symmetric
matrix, has all its roots real. 
It then follows from Descartes' rule of signs 
(see for instance \cite{BPRbook2}),
that for each $(\omega,x) \in \Omega \times \R^k$, 
${\rm index}(\la \omega ,{\mathcal Q}^h\ra (\cdot,x))$ is determined by the sign vector
\[
({\rm sign}(C_\ell(\omega,x)),\ldots,{\rm sign}(C_0(\omega,x))).
\] 
More precisely, the number of  sign variations in the sequence
\[
\s(C_0(\omega,x)),\ldots,(-1)^i \s(C_i(\omega,x)),\ldots, (-1)^\ell \s(C_\ell(\omega,x)),+1
\]
is equal to ${\rm index}(\la \omega , {\mathcal Q}^h \ra (\cdot,x))$.
Hence, denoting 
\begin{equation}
\label{eqn:defincalC}
{\mathcal C} = 
\{C_0,\ldots,C_{\ell}\} \subset \R[Z_1,\ldots,Z_m,X_1,\ldots,X_k],
\end{equation}
we have 

\begin{lemma}
\label{lem:defofD_j}
$F_j$ is the intersection of 
$F$ with a ${\mathcal C}$-closed semi-algebraic set 
$D_j \subset \R^{m+k}$,
for each 
$0 \leq j \leq \ell+1$. \qed
\end{lemma}

\begin{notation}
let
\[
0<\eps_{0} \ll \cdots \ll \eps_{\ell+1} \ll 1.
\]
be infinitesimals.
For $0 \leq j \leq \ell+1$, 
we
denote by $\R_j$ the field $\R\la \eps_{\ell+1} \ldots \eps_j\ra$.

Let 
\[{\mathcal C}'_j=\{P\pm \eps_j, P  \in {\mathcal C}\}. \]

Given $\rho \in {\rm Sign}(\mathcal C)$,
and $0 \leq j \leq \ell+1$, 
we denote by $\RR(\rho^c_j) \subset 
\R_j^{m+k}$ 
the ${\mathcal C}'_j$-semi-algebraic set defined 
by the formula $\rho^c_j$ obtained by taking the 
conjunction of
$$
\begin{array}{l}
 -\eps_j \leq P \leq \eps_j \mbox{ for each } P \in {\mathcal C}
\mbox{ such that } \rho(P) = 0, \cr
P \geq -\eps_{j},  \mbox{ for each } P \in {\mathcal C}
\mbox{ such that } \rho(P) = 1, \cr
P \leq \eps_{j}, \mbox{ for each }  P \in {\mathcal C} 
\mbox{ such that } \rho(P) = -1.
\end{array}
$$

Similarly,
we denote by
$\RR(\rho^o_j) \subset 
\R_j^{m+k}$ 
the ${\mathcal C}'_j$- semi-algebraic set defined 
by the formula $\rho^o$ obtained by taking the 
conjunction of
$$
\begin{array}{l}
-\eps_j < P < \eps_j \mbox{ for each } P \in {\mathcal C}
\mbox{ such that } \rho(P) = 0, \cr
P > - \eps_{j},  \mbox{ for each } P \in {\mathcal C}
\mbox{ such that } \rho(P) = 1, \cr
P < \eps_{j}, \mbox{ for each }  P \in {\mathcal C} 
\mbox{ such that } \rho(P) = -1.
\end{array}
$$
\end{notation}

Since the semi-algebraic sets $D_j$ defined above in 
Lemma \ref{lem:defofD_j} 
are ${\mathcal C}$-semi-algebraic sets, each $D_j$ is defined
by a disjunction of sign conditions on ${\mathcal C}$. More precisely, 
for each $0 \leq j \leq \ell+1$  let $D_j$ be defined by the formula
$$
\displaylines{
D_j = \bigcup_{\rho \in \Sigma_j} \RR(\rho),
}
$$
where $\Sigma_j \subset {\rm Sign}({\mathcal C})$.

For each $j, 0 \leq j \leq \ell+1$, let

\begin{eqnarray*}
D_j^o &=& \bigcup_{\rho \in \Sigma_j} \RR(\rho_j^o), \\
D_j^c &=& \bigcup_{\rho \in \Sigma_j} \RR(\rho_j^c), \\
D_j' &=& \E(D_j^c,
\R_{j-1}) \setminus D_{j-1}^o, \\
F_j' &=& \E(F,
\R_{j-1}) \cap D_j',\\
\end{eqnarray*}
where we denote by $D_{-1}^o = \emptyset $.

\begin{lemma}
\label{lem:local}
For 
$0 \leq j+1 < i \leq \ell+1$,
\[
\E(D_i', 
\R_{j-1}) 
\cap D_j' = \emptyset.
\]
\end{lemma}
\begin{proof}
The inclusions
$$
\displaylines{
D_{j-1} \subset D_j \subset D_{i-1} \subset D_i, \cr
D_{j-1}^o \subset 
\E(D_j^c,
\R_{j-1}) \subset  
\E(D_{i-1}^o,\R_{j-1}) \subset 
\E(D_i^c,
\R_{j-1})
}
$$
follow directly from the definitions of the sets 
\[
D_i,D_j,D_{j-1},D_i^c,D_j^c,D_{i-1}^o, D_{j-1}^o,
\]
and the fact that
\[
\eps_i \gg \eps_{i-1} \gg  \eps_j \gg \eps_{j-1}.
\]

It follows immediately that
\[
D_i' = \E(D_i^c,
\R_{j-1}) \setminus 
\E(D_{i-1}^o,
\R_{j-1})
\]
is disjoint from $\E(D_j^c,
\R_{j-1})$,
and hence also from $D_j'$.
\end{proof}

We now associate to each $F_j'$  
an $\Sphere^{\ell - j}$-bundle as follows.

For each 
$(\omega,x) \in F_j'' = \E(F_j,\R_{j-2})
\setminus F_{j-1}'$, let 
$L_j^+(\omega,x) \subset \R_{j-2}^{\ell+1}$ 
denote the sum of the non-negative eigenspaces of 
$\la \omega , {\mathcal Q}^h \ra (\cdot,x)$ (i.e. $L_j^+(\omega,x)$ is the largest linear
subspace  of $\R_{j-2}^{\ell+1}$ on which $\la \omega , {\mathcal Q}^h \ra (\cdot,x)$ is positive 
semi-definite). 
Since  ${\rm index}(\la \omega ,{\mathcal Q}^h\ra (\cdot,x)) = j$ stays invariant as
$(\omega,x)$ varies over $F_j''$,
$L_j^+(\omega,x)$ varies continuously with $(\omega,x)$.

Let
\[
\lambda_0(\omega,x) \leq \cdots \leq \lambda_{j-1}(\omega,x) < 0 \leq \lambda_j(\omega,x) \leq \cdots \leq \lambda_{\ell}(\omega,x)
\]
be the eigenvalues of $\la \omega , {\mathcal Q}^h \ra(\cdot,x)$ for $(\omega,x) \in 
F_j''$.
There is a continuous extension of the map sending
$(\omega,x) \mapsto L_j^+(\omega,x)$ to 
$(\omega,x) \in \E(F_j',\R_{j-2})$.
To see this observe that for $(\omega,x) \in F_j''$ 
the block of the first $j$ (negative) eigenvalues,
$\lambda_0(\omega,x) \leq \cdots \leq \lambda_{j-1}(\omega,x)$,
and hence the sum of the eigenspaces corresponding to them can be extended
continuously to any infinitesimal neighborhood of 
$F_j''$, and in particular to
$\E(F_j',\R_{j-2})$. Now $L_j^+(\omega,x)$ is the orthogonal
complement of the sum of the eigenspaces 
corresponding to the block of negative eigenvalues,
$\lambda_0(\omega,x) \leq \cdots \leq \lambda_{j-1}(\omega,x)$.

We 
denote by $C_j'\subset F_j' \times \R_{j-1}^{\ell+1}$ 
the semi-algebraic set defined by

\[
C_j'  =  \{(\omega,y;x) \;\mid\; (\omega,x) \in F_j', 
y \in L_j^+(\omega,x), |y| = 1\}.
\]

Abusing notation, we denote by $\varphi_1$ the projection $C_j' \rightarrow F_j'$,
which makes $C_j'$ the total space of a $\Sphere^{\ell - j}$-bundle
over $F_j'$.

The following proposition, expressing in precise terms the fact that
$C'_j\cap C'_{j-1}$ is a  $\Sphere^{\ell-j}$-bundle over 
$F_j \cap F_{j-1}'$ under the map $\varphi_1$,
follows directly
from the definition of the sets $C_j'$ and $F_j'$.

\begin{proposition} 
\label{pro:goodspherebundle}
For every $j$ from $\ell$ to 1,
$C_{j-1}' \cap \E(C_{j}',\R_{j-2})$
is a  $\Sphere^{\ell-j}$-bundle over 
$\E(F_j',\R_{j-2})  \cap F_{j-1}'$ under the map  $\varphi_1$. \qed
\end{proposition}

We also have the following.
\begin{proposition}
\label{pro:homotopy3}
The semi-algebraic set 
$$
C' = \bigcup_{j=0}^{\ell+1} \E(C_j',\R_0)
$$ 
is homotopy equivalent to $\E(C,\R_0)$.
\end{proposition}

\begin{proof}
First observe that $C = \lim_{\eps_{\ell+1}} C'$ where $C$ is the 
semi-algebraic set defined in (\ref{eqn:definition_of_C}) above. 

Now let
$$
\displaylines{
C_{0} = \lim_{\eps_0} C', \cr
C_i = \lim_{\eps_i} C_{i-1}, 1 \leq i \leq \ell+1.
}
$$

Notice that each $C_i$ is a closed and bounded semi-algebraic set.
Also, let $C_{i-1,t} \subset 
\R_{i-1}^{\ell+k}$
be the semi-algebraic set obtained by 
replacing $\eps_i$ in the definition of $C_{i-1}$
by the variable $t$.
Then there exists $t_0 > 0$, such that for all $0 < t_1 < t_2 \leq t_0$,
$C_{i-1,t_1} \subset C_{i-1,t_2}$. 

It follows (see \cite[Lemma 16.17]{BPRbook2}) that for each $i$,
$0 \leq i \leq \ell+1$,
 $\E(C_i,
\R_i)$ is homotopy equivalent to
$C_{i-1}$ (where $C_{-1} = C'$).
 
The proposition is now a consequence of Proposition \ref{pro:homotopy1}.
\end{proof}

\begin{proof}[Proof of Theorem \ref{the:homogeneous}]
In light of Propositions \ref{pro:homotopy1} and \ref{pro:homotopy3},
it suffices to bound the Betti numbers of the semi-algebraic set $C'$.
Now, 
$$
\displaylines{
C' = \bigcup_{j=0}^{\ell+1} \E(C_j',\R_0).
}
$$ 

By \eqref{eqn:MV1} it suffices to
bound the Betti numbers of the various intersections amongst the sets
$\E(C_j',\R_0)$'s. 
However, by Lemma \ref{lem:local}, the only non-empty intersections
among $\E(C_j',\R_0)$'s are of the form $\E(C_j',\R_0) \cap \E(C_{j+1}',\R_0)$. 
Using Proposition \ref{pro:sphere_bundle} and Proposition \ref{pro:goodspherebundle}
we have that 
$b(C_j')$ (resp.  $b(C_j' \cap C_{j+1}')$) is bounded by
$2 b(F_j')$ (resp.  $2 b(\E(F_j',\R_0) \cap \E(F_{j+1}',\R_0))$).

Finally, each $F_j'$ (resp. $\E(F_j',\R_0) \cap \E(F_{j+1}',\R_0)$) is a bounded 
${\mathcal P}'_j$-closed semi-algebraic set, where
${\mathcal P}'_j = \R[Z_1,\ldots,Z_m,X_1,\ldots,X_k]$ is defined by
$$
\displaylines{
{\mathcal P}'_j = {\mathcal P} \cup {\mathcal C}'_j\cup  
\bigcup_{i=1}^{m}  \{Z_i\}.
}
$$
Note that 
$$
\displaylines{
\deg(P) \leq d, P \in {\mathcal P}, \cr
\deg(P) \leq d(\ell+1),  P \in {\mathcal C}'_j, \cr
\#({\mathcal P}) = s, \cr
\#({\mathcal C}'_j) = 2(\ell+1).
}
$$

Now applying Theorem \ref{the:B99} we obtain that 
\begin{equation}
\label{eqn:boundonF_j}
b(F_j'), b(\E(F_{j}',\R_0) \cap \E(F_{j+1}',\R_0)) \leq 
(O((s+\ell+m)\ell d))^{m+k}.
\end{equation}

Applying Proposition \ref{pro:sphere_bundle} and (\ref{eqn:boundonF_j}) 
we obtain immediately that

\begin{equation}
\label{eqn:boundonC_j}
b(C_j'), b(\E(C_{j}',\R_0) \cap \E(C_{j+1}',\R_0)) \leq 
(O((s+\ell+m)\ell d))^{m+k}.
\end{equation}

Finally, using inequality \eqref{eqn:MV1} and 
Lemma \ref{lem:local} we get that

\begin{equation}
\label{eqn:boundonC}
b(C') = b(\bigcup_{j=0}^{\ell} C_j') \leq 
\ell^2 (O((s+\ell+m)\ell d))^{m+k}.
\end{equation}

The theorem now follows from Propositions \ref{pro:homotopy3},
\ref{pro:homotopy2} and \ref{pro:homotopy1}.
\end{proof}

\begin{proof}[Proof of Theorem \ref{the:homogeneous2}]
Apply \eqref{eqn:MV2} together with Theorem \ref{the:homogeneous}.
\end{proof}

\subsection{General Case}
We now prove the general version of Theorem \ref{the:homogeneous}.
We follow Notation \ref{not:AhWh}.

\begin{theorem}
\label{the:inhomogeneous}
Let 
$W \subset \R^{\ell} \times \R^k$ be semi-algebraic set defined by
$$
\displaylines{
W = \bigcap_{Q \in {\mathcal Q}}
\{ (y,x) \;\mid\; Q(y,x) \leq 0\; \wedge \; \Phi(x)\},
}
$$
where $\Phi(x)$ is a ${\mathcal P}$-closed formula defining a bounded
${\mathcal P}$-closed semi-algebraic set $V \subset \R^k$.

Then 
\begin{eqnarray}
b(W) &\leq &
\ell^2 (O((s+\ell+m)\ell d))^{m+k}.
\end{eqnarray}
\end{theorem}

\begin{proof}
Let $1 \gg \eps > 0$ be an infinitesimal and let
$B_\ell(0,{1}/{\eps})$ denote the closed ball in 
$\R\la\eps\ra^\ell$ centered at the origin and of radius ${1}/{\eps}$.

Let $W_\eps \subset \R^{\ell+k}$ be the set defined by 
\begin{eqnarray*}
W_\eps &=& W \cap \left( B_\ell(0,{1}/{\eps}) \times \R^k\right)  
\end{eqnarray*}

It follows from the local conical structure of semi-algebraic sets at 
infinity \cite[Theorem~9.3.6]{BCR} that $W_\eps$ 
has the same homotopy type as $\E(W,\R\la\eps\ra)$.

Let
$$
Q_0 = \eps^2(Y_1^2 + \cdots + Y_\ell^2) - 1,
$$ and
$W_\eps^h \subset \Sphere^{\ell}\times \R\la\eps\ra^k$ be the
semi-algebraic set defined by

$$
\displaylines{
W_\eps^h = \bigcap_{i=0}^{m}
\{ (y,x) \;\mid\; |y| = 1 \;\wedge \; Q_i^h(y,x) \leq 0\; \wedge \; \Phi(x)\}.
}
$$

It is clear that $W_\eps^h$ is a union of two disjoint,
closed and bounded semi-algebraic sets, each homeomorphic to $W_\eps$. 
Hence, for every $i=0,\ldots,k+\ell-1$
\begin{equation}
\label{eqn:doubling}
b_i(W_\eps^h) =2 b_i(W_\eps) =2 b_i(W).
\end{equation}
The theorem is proved
by applying Theorem \ref{the:homogeneous2} to $W_\eps^h$.
\end{proof}

\subsection{Proof of Theorem \ref{the:main}}
We are now in a position to prove Theorem \ref{the:main}. 
We first need a few preliminary results.

Given a
list of polynomials ${\mathcal A}=\{A_1,\ldots,A_t\}$
with coefficients in 
$\R$, we introduce $t$ infinitesimals,
$1 \gg \delta_1 \gg \cdots \gg \delta_t > 0$.

We define ${\mathcal A}_{>i}=\{ A_{i+1},\ldots, A_t\}$
and
$$ \displaylines{
\Sigma_i= \{A_i=0,A_i=\delta_i,A_i=-\delta_i,A_i\geq
2\delta_i,A_i\leq -2\delta_i\},\cr
\Sigma_{\le i}= \{\Psi \mid
\Psi=\bigwedge_{j=1,\ldots,i} \Psi_i, \Psi_i \in \Sigma_i\}.
}$$
If $\Phi$ is any
${\mathcal A}$-closed formula, 
we denote
by
$\RR_i(\Phi)$ the   extension of
$\RR(\Phi)$ to \\ 
$\R\la \delta_1,\ldots, \delta_i \ra ^k$.
For $\Psi \in \Sigma_{\le i}$, we denote by $\RR_i(\Psi)$ the realization of
$\Psi$ and by $b(\Psi)$ the sum of the Betti numbers of $\RR_i(\Psi)$ .

\begin{proposition}
\label{7:pro:closed}
For every 
${\mathcal A}$-closed 
formula $\Phi$,
$$
b(\Phi)  \leq
\sum_{\Psi \in \Sigma_{\le t}} b(\Psi).
$$
\end{proposition}

\begin{proof}
See \cite[Proposition 7.39]{BPRbook2}.
\end{proof}

\begin{proof}[Proof of Theorem \ref{the:main}]
First note that we can assume (if necessary by adding  to ${\mathcal Q}$
an extra quadratic inequality)
that the set $S$ is bounded.

Denoting ${\mathcal P}=\{P_1,\dots,P_s\}$, 
define ${\mathcal B} = \{B_1,\ldots,B_{s+m} \}$, where
\[
B_i =
\begin{cases}
Q_i,& 1 \leq i \leq m,  \\
P_{i-m},& m+1 \leq i \leq m+s.
\end{cases}
\]
It follows from Proposition \ref{7:pro:closed} that in order to bound
$b(S)$, it suffices to bound $b(T)$,  where $T$ is defined by
$$
\displaylines{
\bigwedge_{i=1}^{s+m} 
B_i^2(B_i^2 - \delta_i^2)^2 (B_i^2 - 4\delta_i^2)\geq 0.
}
$$

We now introduce $m$ new variables, $Z_1,\ldots,Z_m$ and let
\[
{\mathcal A} = \{A_1,\ldots,A_{s+m}\} 
\subset \R[Y_1,\ldots,Y_\ell,X_1,\ldots,X_k,Z_1,\ldots,Z_m]
\]
be defined by
\[
A_i =
\begin{cases} Z_i,& 1 \leq i \leq m,  \\
   P_{i-m},& m+1 \leq i \leq m+s.
\end{cases}
\]
Consider the semi-algebraic set $T' \subset \R^{m+k+l}$ defined by
$$
\displaylines{
\bigwedge_{i=1}^{s+m} 
A_i^2(A_i^2 - \delta_i^2)^2 (A_i^2 - 4\delta_i^2)\geq 0
\wedge 
\bigwedge_{i=1}^{m} (Z_i - Q_i = 0).
}
$$
Clearly, $T$ is homeomorphic to $T'$. Notice that the number of polynomials
in the definition of $T'$, which depend only on 
$X$
and $Z$ is $s+m$, and
the degrees of these polynomials are bounded by  
$6d$. 
The number of polynomials
depending on $X,Y$ and $Z$ is $m$ and these are of degree at most $2$
in $Y$ and at most $d$ in the remaining variables. Thus, we are in a position
to apply Theorem \ref{the:inhomogeneous} to obtain that
\[
b(S) \leq b(T')   \leq  \ell^2 (O(s+\ell+m)\ell d)^{k+2m}. 
\]
This proves the theorem.
\end{proof}

\begin{proof}[Proof of Corollary \ref{cor:main}]
Introduce $k$ new variables, $Z_1,\ldots,Z_k$, and let
$\tilde{Q}_i = Z_i - Q_i$ for $1 \leq i \leq k$.

Define the semi-algebraic set $\tilde{S} \subset \R^{\ell+k}$ by
$$
\displaylines{
\tilde{S} = \{ (y,x) \;\mid\; \bigwedge_{i=1}^k \tilde{Q}_i(y,x) = 0 \wedge \Phi(x)\}.
}
$$
It is clear that $\tilde{S}$ is semi-algebraically homeomorphic to $S$.
Applying Theorem \ref{the:main} to $\tilde{S}$, we obtain the desired bound. 
\end{proof}

\section{Algorithm for Computing the Euler-Poincar\'e characteristic}
\label{sec:algo_ep}
We first need a few preliminary definitions and results.

\subsection{Some Algorithmic and Mathematical Preliminaries}
Recall that for 
a closed and bounded semi-algebraic set $S \subset \R^k$, the 
Euler-Poincar\'e characteristic of $S$, denoted by $\chi(S)$, 
is defined by
\[
\chi(S) = \sum_{i=0}^k 
(-1)^i~ b_i(S).
\]
Moreover, we have the following additivity property which is classical.
\begin{proposition}
\label{6:pro:additivityforordinary}
Let $X_1$ and $X_2$ be closed and bounded semi-algebraic
sets.  Then
$$
\chi(X_1\cap X_2)=\chi(X_1)+\chi(X_2) - \chi(X_1 \cup X_2).
$$
\end{proposition}

Recall 
also that for 
a locally closed semi-algebraic set $S$, the Borel-Moore
Euler-Poincar\'e characteristic of $S$, denoted by $\chi^{BM}(S)$, 
is defined by
\[
\chi^{BM}(S) = \sum_{i=0}^k 
(-1)^i~
b_i^{BM}(S),
\]
where $b_i^{BM}(S)$ denotes the dimension of the $i$-th Borel-Moore
homology group $\HH_i^{BM}(S,\Z/2\Z)$ of $S$.
Note that $\chi^{BM}(S) = \chi(S)$ for $S$ closed and bounded.
 
Note that $\chi^{BM}(S)$ has the following classically known
(see e.g.  \cite{BPRbook2} for a proof) additivity property.

\begin{proposition}
\label{6:pro:additivity}
Let $X_1$ and $X_2$ be locally closed semi-algebraic
sets such that  $X_1\cap X_2=\emptyset.$  Then
$$\chi^{BM}(X_1\cup X_2)=\chi^{BM}(X_1)+\chi^{BM}(X_2),$$
provided that $X_1\cup X_2$ is locally closed  as well. 
\end{proposition}

Let $Z\subset \R^k$ and $Q\in \R[X_1,\ldots,X_k]$. We define
\begin{align*}
\RR(Q=0,Z)&= \{x\in Z\;\mid\;
Q(x)=0 \},\\
\RR(Q>0,Z)&= \{x\in Z\;\mid\;
Q(x)>0 \},\\
\RR(Q<0,Z)&= \{x\in Z\;\mid\;
Q(x)<0 \}.
\end{align*}

\begin{corollary}
\label{6:cor:additivity}
Let $Z \subset \R^k$ be a locally closed semi-algebraic
set. Then
$$\chi^{BM}(Z)=\chi^{BM}(\RR(Q=0,Z))+\chi^{BM}(\RR(Q>0,Z))+\chi^{BM}(\RR(Q<0,Z)).$$
\end{corollary}

\begin{notation}
\label{not:ep}
Let $Z \subset \R^k$ be a locally closed semi-algebraic set and
let ${\mathcal A}$
be a finite subset of $\R[X_1,\ldots,X_k]$.

The  realization of the  sign condition
${\rho}\in \{0,1,- 1\}^{\mathcal A}$ 
on  ${Z}$  is
$$
\RR(\rho,Z)= \{x\in Z\;\mid\;
 \bigwedge_{A \in {\mathcal A}} \s({A}(x))=\rho(A) \},
$$
and its Borel-Moore
Euler-Poincar\'e characteristic is denoted $\chi^{BM}(\rho,Z).$

We denote  by ${\rm Sign}({\mathcal A},Z)$   the list of
$\rho \in  \{0,1,- 1\}^{\mathcal A}$
such that $\RR(\rho,Z)$ is non-empty.
We denote by $\chi^{BM}({\mathcal A},Z)$  the  list of
Euler-Poincar\'e characteristics \\
$\chi^{BM}(\rho,Z)=\chi^{BM}(\RR(\rho,Z))$ for
$\rho \in {\rm Sign}({\mathcal A},Z)$.

Finally, 
given two finite families of polynomials, ${\mathcal A} \subset {\mathcal A}'$,
and $\rho \in \{0,1,-1\}^{{\mathcal A}}, \rho' \in \{0,1,-1\}^{{\mathcal A}'}$,
we define $\rho \prec \rho'$ 
by:
for all $P \in {\mathcal A},$ $\rho(P) = \rho'(P)$.
\end{notation}

We will use the following algorithm for computing the list
$\chi^{BM}({\mathcal A},Z)$ described in \cite{BPRbook2}. 
We  
recall
here the
input, output and complexity of the algorithm.

\begin{algorithm}[Euler-Poincar\'e Characteristic of Sign Conditions]
\label{13:alg:eulerdet}
\item[]
\item[{\sc Input}] 
A finite list ${\mathcal A}= \{A_1,\ldots,A_t\}$ of polynomials  in
$\R[X_1,\ldots,X_k]$.
\item[{\sc Output}] 
The list $\chi^{BM}({\mathcal A})$.
\end{algorithm}

\medskip
\noindent\textsc{Complexity:}
Let $d$ be a bound on the degrees of the polynomials
in ${\mathcal A}$,
and $t=\#({\mathcal A})$.
The number of arithmetic operations is bounded by 
$$t^{k+1}O(d)^k+ t^{k}((k\log_2(s)+k\log_2(d))d)^{O(k)}.$$
The algorithm also involves the inversion  of matrices of
size $t^{k}O(d)^k$ with integer coefficients.

\subsection
{Algorithms for the Euler-Poincar\'e characteristic}

We first deal with the special case of 
polynomials which are homogeneous and of degree two
in the variables $Y_0,\ldots,Y_\ell$,
and in this case we describe algorithms 
(Algorithms \ref{alg:union} and \ref{alg:homogeneous} below)
for computing the Euler-Poincar\'e
characteristic of the sets $A^h$ and $W^h$ respectively.
We then use Algorithm \ref{alg:homogeneous} to derive 
algorithms for computing the Euler-Poincar\'e
characteristic in the 
general case 
(Algorithms \ref{alg:basic} and  \ref{alg:general} below).

\subsubsection{Homogeneous quadratic polynomials}

\begin{algorithm} [Euler-Poincar\'e characteristic, homogeneous union case]
\label{alg:union}
\item[]
\item[{\sc Input}]
\item[]
\begin{itemize}
\item
A family of polynomials,
$
{\mathcal Q}^h  \subset  \R[Y_0,\ldots,Y_\ell,X_1,\ldots,X_k],
$
with $\deg_{Y}(Q) \leq 2, 
\deg_{X}(Q) \leq d, 
Q \in {\mathcal Q}^h, \#({\mathcal Q}^h)=m$, homogeneous with respect to 
$Y$,
\item
another family, 
${\mathcal P}
\subset \R[X_1,\ldots,X_k]$ with
$ \deg_{X}(P) \leq d, P \in {\mathcal P},  \#({\mathcal P})=s$,
\item
a formula $\Phi$ defining a bounded ${\mathcal P}$-closed semi-algebraic set $V$.
\end{itemize}

\item[{\sc Output}]
the Euler-Poincar\'e characteristic $\chi(A^h)$, 
where $A^h$ is
the semi-algebraic set defined by
$$
\displaylines{
A^h = \bigcup_{Q \in {\mathcal Q}^h }
\{ (y,x) \;\mid\; |y|=1\; \wedge\; Q(y,x) \leq 0\; \wedge \; \Phi(x)\}.
}
$$
\item[{\sc Procedure}]
\item[] 
\item[Step 1.]
Let $Z= (Z_1,\ldots,Z_m)$ be variables and let $M$ be the symmetric
matrix with entries in $\R[Z_1,\ldots,Z_m,X_1,\ldots,X_k]$ associated
to the quadratic form $\la Z , {\mathcal Q}^h\ra$. 
Obtain $C_i \in \R[Z_1,\ldots,Z_m,X_1,\ldots,X_k]$ 
by computing the following determinant.
\[
\det( T \cdot {\rm Id}_{\ell+1}-M)= T^{\ell+1} + 
C_{\ell}T^{\ell} + \cdots+ C_0.
\]
\item[Step 2.]
Compute $\chi^{BM}({\mathcal C},F)$ as follows.
Call Algorithm \ref{13:alg:eulerdet} with input 
${\mathcal C}' = {\mathcal C} 
\cup {\mathcal P}$.
Compute from the output the list 
$
\chi^{BM}({\mathcal C}, F),
$
using the additivity property 
of the Borel-Moore Euler-Poincar\'e characteristic (Proposition 
\ref{6:pro:additivity}).
For each $\rho \in \{0,+1,-1\}^{{\mathcal C}}$, such that there exists
$\rho' \in {\rm Sign}({\mathcal C'}, F)$ with 
$\rho \prec \rho'$ (see Notation \ref{not:ep}) 
and $\rho'(Z_j) \in \{0,-1\}$ for $1 \leq j \leq m,$
compute
$$
\chi^{BM}(\rho,F) = 
\sum_{\substack{\rho',\rho \prec \rho',\\
\rho'(Z_j) \in \{0,-1\},1 \leq j \leq s}}
\chi^{BM} (\rho',F).
$$

\item[Step 3.]
Output 
\[
\chi(A^h) = 
\sum_{\rho \in {\rm Sign}({\mathcal C},F)} \chi^{BM}(\RR(\rho,
F))
\cdot (1 + (-1)^{(k - n(\rho))}),
\]
where $n(\rho)$ denotes the number of sign variations in the sequence,
\[
\rho(C_0),\ldots,(-1)^i \rho(C_i),\ldots, (-1)^\ell \rho(C_\ell),+1.
\]
\end{algorithm}

\medskip
\noindent\textsc{Proof of Correctness:}
It follows from Lemma \ref{lem:sphere} that 
for 
any $\rho \in {\rm Sign}({\mathcal C},F)$
\[
\chi^{BM}(\varphi_1^{-1}(\RR(\rho))) = 
\chi^{BM}(\RR(\rho))\cdot(1 + (-1)^{(k - n(\rho))}).
\]
Also, by virtue of Proposition \ref{pro:homotopy2} we have that 
$$
\chi^{BM}(B) = \chi(A^h),
\qquad\text{where}\quad 
B = \bigcup_{\rho \in {\rm Sign}({\mathcal C},F)} \varphi^{-1}(\RR(\rho)).
$$
The correctness of the algorithm is now a consequence of 
the additivity property of the Borel-Moore Euler Poincar\'e characteristic
(Proposition \ref{6:pro:additivity})
and the correctness of Algorithm \ref{13:alg:eulerdet}.
\qed

\medskip
\noindent\textsc{Complexity Analysis:}
The complexity of the algorithm is  $(\ell s m d)^{O(m+k)}$ 
using the complexity of 
Algorithm \ref{13:alg:eulerdet}.
\qed

We are now in a position to describe the algorithm for computing the
Euler-Poincar\'e characteristic in the homogeneous intersection  case.

\begin{algorithm} [Euler-Poincar\'e characteristic, homogeneous intersection case]
\label{alg:homogeneous} 
\item[]
\item[{\sc Input}]
\item[]
\begin{itemize}
\item
A family of polynomials,
$
{\mathcal Q}^h = \{Q_1^h,\ldots, Q_m^h\} \subset  \R[Y_0,\ldots,Y_\ell,X_1,\ldots,X_k],
$
with $\deg_{Y}(Q) \leq 2, 
\deg_{X}(Q) \leq d, 
Q \in {\mathcal Q}^h$, homogeneous with respect to $Y$,
\item
another family, 
${\mathcal P} 
\subset \R[X_1,\ldots,X_k]$ with
$\deg_{X}(P) \leq d, P \in {\mathcal P}, \#({\mathcal P})=s$,

\item
a formula $\Phi$ defining a bounded 
${\mathcal P}$-closed semi-algebraic set $V$.
\end{itemize}

\item[{\sc Output}]
\item[] 
the Euler-Poincar\'e characteristic $\chi(W^h)$,
where $W^h$ is the semi-algebraic set defined by
$$
\displaylines{
W^h = \bigcap_{Q\in {\mathcal Q}^h}
\{ (y,x) \;\mid\; |y|=1\; \wedge\; Q(y,x) \leq 0\; \wedge \; \Phi(x)\}.
}
$$

\item[{\sc Procedure}]
\item[] 
\item[Step 1.]
For each subset $J \subset [m]$ do the following.

\item
Compute 
$\chi(A^J)$
using Algorithm \ref{alg:union}, where
$$
\displaylines{
A^{J} = \bigcup_{Q \in J}
\{ (y,x) \;\mid\; |y|=1\; \wedge\; Q(y,x) \leq 0\; \wedge \; \Phi(x)\}.
}
$$

\item [Step 2.]
Output 
\begin{equation}
\label{eqn:W^h}
\chi(W^h) = \sum_{J \subset {\mathcal Q}} (-1)^{\#(J)+1} \chi(A^{J}).
\end{equation}

\end{algorithm}

\medskip
\noindent\textsc{Proof of Correctness:}
First note that the equality \ref{eqn:W^h} can be easily deduced from
Proposition \ref{6:pro:additivityforordinary} by induction.
The correctness of the algorithm is now 
a consequence of the correctness of  Algorithm \ref{alg:union}.
\qed

\medskip
\noindent\textsc{Complexity Analysis:}
There are $2^{m}$ calls to  Algorithm \ref{alg:union}. Using the
complexity analysis of  Algorithm \ref{alg:union}, the complexity of
the algorithm is bounded by $(\ell s m d)^{O(m+k)}.$
\qed
\subsubsection{The Case of 
Intersections}
\begin{algorithm} [Euler-Poincar\'e Characteristic, 
Intersection Case]
\label{alg:basic}
\item[{\sc Input}]
\item[]
\begin{itemize}
\item A family of polynomials, ${\mathcal Q} \subset \R[Y_1,\ldots,Y_\ell,X_1,\ldots,X_k],$
with
$
\deg_{Y}(Q) \leq 2, 
\deg_{X}(Q) \leq d, 
Q \in {\mathcal Q}, \#({\mathcal Q})=m
$
\item another family of polynomials,
$
{\mathcal P} 
\subset \R[X_1,\ldots,X_k]
$
with 
$\deg_{X}(Q) \leq d,  P \in {\mathcal P},  \#({\mathcal P})=s$,
\item
a ${\mathcal P}$-closed formula 
$\Phi$ defining a  ${\mathcal P}$-closed 
semi-algebraic set $V \subset \R^k$.
\end{itemize}
\item [{\sc Output}]
the Euler-Poincar\'e characteristic $\chi(W)$,
where $W$ is the semi-algebraic set defined by
$$
\displaylines{
W = \bigcap_{Q \in {\mathcal Q}}
\{ (y,x) \;\mid\; Q(y,x) \leq 0\; \wedge \; \Phi(x)\}.
}
$$

\item [{\sc Procedure}]
\item[]
\item[Step 1.]
Replace 
${\mathcal Q}^h$ 
by
${\mathcal Q}^h \cup \{Q_0^h \},$
with 
$Q_0=\eps^2(Y_1^2+\ldots+Y_\ell^2)-1$.
Define
$$
\displaylines{
W_\eps^h = \bigcap_{Q^h \in {\mathcal Q}^h}
\{ (y,x) \;\mid\; |y|=1\; \wedge\; Q^h(y,x) \leq 0\; \wedge \; \Phi(x)\}.
}
$$
\item [Step 2.]
Using Algorithm \ref{alg:homogeneous} compute 
$\chi(W_\eps^h)$.
\item [Step 3.]
Output $\chi(W) = \frac{1}{2}\chi(W_\eps^h)$.
\end{algorithm}

\medskip
\noindent\textsc{Proof of Correctness:}
The correctness of  Algorithm \ref{alg:basic} 
follows from \eqref{eqn:doubling}
and the correctness of Algorithm \ref{alg:homogeneous}.
\qed

\medskip
\noindent\textsc{Complexity Analysis:}
The complexity of the algorithm is clearly $(\ell s m  d)^{O(m+k)}$ arithmetic operations in $\R\la \eps \ra$
from the complexity analysis of Algorithm  \ref{alg:homogeneous}.
Moreover the maximum degree in $\eps$ is bounded by $(\ell m  d)^{O(m+k)}$.
Finally the
complexity of the algorithm is $(\ell s m  d)^{O(m+k)}$ arithmetic operations in $\R$.
\qed
\subsubsection{The case of a  ${\mathcal Q} \cup {\mathcal P}$-closed semi-algebraic set}
\label{sec:general}

Since we want to deal with a general 
${\mathcal Q} \cup {\mathcal P}$-closed semi-algebraic set, 
we shall need a property similar to 
Corollary \ref{6:cor:additivity}
in a context where all the sets considered are closed and bounded.

We need a few preliminary definitions and results.
Let  ${\mathcal Q}=\{Q_1,\ldots,Q_m\}$ 
and
\[
0<\eps_m \ll \cdots \ll \eps_1 \ll \eps_0 \ll 1
\]
be infinitesimals.
For every $j \in [m]=\{1,\ldots,m\}$,  denote 
$\R_j=\R\la \eps_{0},\ldots,\eps_j\ra$.
Let
\begin{align*}
\Psi_i^0 &= (Q_i=0), \\
\Psi_i^1 &= (Q_i\ge \eps_i), & \Psi_i^{-1} &= (Q_i\le -\eps_i), \\
\Psi_i^{2} &=  (Q_i = \eps_i), & \Psi_i^{-2} &= (Q_i =-\eps_i).
\end{align*}

The following Lemma~\ref{lem:generalizedsigns} plays a role similar to  
Corollary \ref{6:cor:additivity}.
\begin{lemma}
\label{lem:generalizedsigns}
Let $S$ be a ${\mathcal Q} \cup {\mathcal P}$-closed bounded 
semi-algebraic set. For every $j\in [m]$
\[\chi(S)=
\chi(\RR(\Psi_j^{0},S))+\chi(\RR(\Psi_j^{1},S))+\chi(\RR(\Psi_j^{-1},S))-\chi(\RR(\Psi_j^{2},S))
-\chi(\RR(\Psi_j^{-2},S))
\]
\end{lemma}
\begin{proof} 
The claims follow from the additivity property
of the Euler-Poincar\'e characteristic, 
and the fact that
\[
\chi(\RR(\Psi_j^{0},S))=\chi(\{(x,y)\in S \mid -\eps_j\le 
Q_j(x,y)\le \eps_j\}),\] since
$\RR(\Psi_j^{0},S)$ is a deformation retract of $\{(x,y)\in S \mid -\eps_j\le Q_j(x,y)\le \eps_j\}$.
\end{proof}

We define $\Sigma_m= \{-2,-1,0,1,2\}^{[m]}$.
Given $\rho \in \Sigma_m$ we define
\[
\RR(\rho,S)=\{(x,y)\in \E(S,\R_m)\mid 
\bigwedge_{i=1}^m\Psi_i^{\rho(i)}(x,y)
\}.
\]

For any  $\rho \in \Sigma_m$ and $\sigma$ a weak sign condition on 
${\mathcal Q} \cup {\mathcal P}$, we say that $\rho \prec \sigma$, if
for each $i \in [m]$, $\s(\rho(i)) \in \sigma(Q_i)$ and $\R(\sigma) \subset S$.

Notice that an alternative description of $\RR(\rho,S)$ is given by
\begin{equation}
\label{eqn:defofR(rho,S)}
\RR(\rho,S)=\left\{
(x,y)\in \R_m^{\ell+k}\mid 
\bigwedge_{i=1}^m\Psi_i^{\rho(i)}(x,y)
\wedge
\left(\bigvee_{\rho \prec\sigma}
\bigwedge_{P \in {\mathcal P}} (\s(P(x) \in \sigma(P))
\right)
\right\}.
\end{equation}

\begin{algorithm} [Euler-Poincar\'e, the general case]
\label{alg:general}
\item[{\sc Input}]
\item[]
\begin{itemize}
\item A family of polynomials, ${\mathcal Q} = \{Q_1,\ldots,Q_m\} \subset \R[Y_1,\ldots,Y_\ell,X_1,\ldots,X_k]$,
with
$
\deg_{Y}(Q) \leq 2, \deg_{X}(Q) \leq d,
$
\item
another family of polynomials,${\mathcal P} 
\subset \R[X_1,\ldots,X_k]$
with 
$\deg_{X}(P) \leq d, P \in {\mathcal P}, \#({\mathcal P})=s$,
\item a ${\mathcal Q} \cup {\mathcal P}$-closed 
formula defining a ${\mathcal Q} \cup {\mathcal P}$-closed 
semi-algebraic set $S$.

\end{itemize}
\item [{\sc Output}]
the Euler-Poincar\'e characteristic $\chi(S)$.

\item [{\sc Procedure}]
\item[]
\item[Step 1.]
Define  $Q_0=\eps_0^2(Y_1^2+\ldots+Y_\ell^2)-1$, 
$P_0=\eps_0^2(X_1^2+\ldots+X_k^2)-1$.
Replace ${\mathcal P}$ by ${\mathcal P}\cup \{P_0\}$ and $S$ by $\RR(S,\R \la \eps \ra)\cap (\RR(Q_0 \le 0) \times \RR(P_0 \le 0))$.

\item[Step 2.]

For every generalized sign condition $\rho \in \Sigma_m$ compute 
$\chi(\RR(\rho,S))$
using \ref{eqn:defofR(rho,S)} and Algorithm \ref{alg:basic}.

\item [Step 3.]
Denoting by $n(\rho)=\#(\{i\in [m] \mid \vert \rho(i)\vert =2\} )$, output 
\[\chi(S) =  \sum_{\rho} (-1)^{n(\rho)}\chi(\RR(\rho,S)).
\]
\end{algorithm}

\medskip
\noindent\textsc
{Proof of Correctness:}
It follows from the local conic structure of semi-algebraic sets at 
infinity \cite[Theorem~9.3.6]{BCR} that
replacing $S$ by 
$\RR(S,\R \la \eps \ra)\cap (\RR(Q_0 \le 0) \times \RR(P_0 \le 0))$
does not modify the Euler-Poincar\'e characteristic.
The proof is now based on the following 
lemma.

\begin{lemma}
Let $S$ be a a ${\mathcal Q} \cup {\mathcal P}$-closed and bounded semi-algebraic set.
Denoting by $n(\rho)=\#(\{i\in [m] \mid \vert \rho(i)\vert =2\} )$, for $\rho \in \Sigma_m$,
\[\chi(S) =  \sum_{\rho \in \Sigma_m} (-1)^{n(\rho)}\chi(\RR(\rho,S)).
\]
\end{lemma}
\begin{proof}
The proof is by induction on $m$.
The induction hypothesis $H_j$ states that 
denoting by $n(\rho)=\#(\{i\in [j]\mid \vert \rho(i)\vert =2\} )$ for $\rho \in \Sigma_j$,
\[\chi(S) =  \sum_{\rho \in \Sigma_j} (-1)^{n(\rho)}\chi(\RR(\rho,S)).
\]
The base case $H_1$ is exactly Lemma \ref{lem:generalizedsigns} applied to $S$.
Suppose now 
that $H_{j-1}$ holds for 
some 
$1<j\le m$, i.e.
\begin{equation}
\label{eqn:foreuler}
\chi(S) =  \sum_{\rho \in \Sigma_{j-1}} (-1)^{n(\rho)}\chi(\RR(\rho,S)) 
\end{equation}

and let us prove $H_{j}$.
Define ${\mathcal Q}_j={\mathcal Q}\cup \{Q_i\pm\eps_i, i=1,\ldots,j\}$.

For every $\rho \in \Sigma_{j-1}$, $\RR(\rho,S)$ is a ${\mathcal Q}_{j-1} \cup {\mathcal P}$-closed semi-algebraic set.
Denoting by  $\rho_i \in \Sigma_j$, for $\rho \in \Sigma_{j-1}$, $i\in \{-2,-1,0,1,2\}$, the generalized sign condition defined by $\rho_i(u)=\rho(u)$, $u=1,\ldots,j-1$, $\rho_i(j)=i$, notice that
$\RR(\rho_i,S)=\RR(\Psi_j^{i},\RR(\rho,S))$.
Using Lemma \ref{lem:generalizedsigns} applied to  $\RR(\rho,S)$, we obtain
\begin{multline*}
\chi(\RR(\rho,S))=\chi(\RR(\rho_0,S))+\chi(\RR(\rho_1,S))+\chi(\RR(\rho_{-1},S))
\\
-\chi(\RR(\rho_2,S))-\chi(\RR(\rho_{-2},S))).
\end{multline*}

Substituting each $\chi(\RR(\rho,S))$ by its value in (\ref{eqn:foreuler}) one gets $H_j$, since every element of $\Sigma_j$ is of the form $\rho_i$ for some $\rho \in \Sigma_{j-1}$, $i\in \{-2,-1,0,1,2\}$.
\end{proof}
The correctness of Algorithm \ref{alg:general} now follows directly from the
previous lemma.
\qed

\medskip
\noindent\textsc
{Complexity Analysis:}
There are $5^m$ calls to Algorithm  \ref{alg:basic}.
The complexity of the algorithm is clearly $(\ell s m  d)^{O(m+k)}$ arithmetic operations in $\R_m$
from the complexity analysis of Algorithm  \ref{alg:basic}.
Moreover the maximum degree in $\eps_0,\ldots,\eps_m$ is bounded by $(\ell m  d)^{O(m+k)}$.
Finally the
complexity of the algorithm is $(\ell s m  d)^{O(m(m+k))}$ arithmetic operations in $\R$.
\qed

\begin{proof}[Proof of Theorem \ref{the:algo-EP}]
The proof of correctness and the complexity analysis of Algorithm
\ref{alg:general} also proves Theorem \ref{the:algo-EP}.
\end{proof}

\bibliographystyle{amsplain}
\bibliography{master}

\def\cprime{$'$} \def\cprime{$'$}
\providecommand{\bysame}{\leavevmode\hbox to3em{\hrulefill}\thinspace}
\providecommand{\MR}{\relax\ifhmode\unskip\space\fi MR }
\providecommand{\MRhref}[2]{%
  \href{http://www.ams.org/mathscinet-getitem?mr=#1}{#2}
}
\providecommand{\href}[2]{#2}
\begin{thebibliography}{10}

\bibitem{Agrachev}
A.A. Agrachev, \emph{{Topology of quadratic maps and Hessians of smooth maps}},
  Algebra, Topology, Geometry, Itogi Nauki i Tekhniki, Akad. Nauk SSSR,
  Vsesoyuz. Inst. Nauchn. Tekhn. Inform., vol.~26, VINITI, Moscow, 1988,
  Translated in J. Soviet Mathematics. 49 (1990), no. 3, 990-1013., pp.~85--124
  (Russian, English).

\bibitem{Bar93}
A.~I. Barvinok, \emph{Feasibility testing for systems of real quadratic
  equations}, Discrete Comput. Geom. \textbf{10} (1993), no.~1, 1--13.
  \MR{94f:14051}

\bibitem{Bar97}
\bysame, \emph{On the {B}etti numbers of semialgebraic sets defined by few
  quadratic inequalities}, Math. Z. \textbf{225} (1997), no.~2, 231--244.
  \MR{98f:14044}

\bibitem{B99}
S.~Basu, \emph{On bounding the {B}etti numbers and computing the {E}uler
  characteristic of semi-algebraic sets}, Discrete Comput. Geom. \textbf{22}
  (1999), no.~1, 1--18.

\bibitem{Basu_survey}
\bysame, \emph{Algorithmic semi-algebraic geometry and topology -- recent
  progress and open problems}, Surveys on Discrete and Computational Geometry:
  Twenty Years Later, Contemporary Mathematics, vol. 453, American Mathematical
  Society, 2008, pp.~139--212.

\bibitem{Bas05-top}
\bysame, \emph{Computing the top few {B}etti numbers of semi-algebraic sets
  defined by quadratic inequalities in polynomial time}, Found. Comput. Math.
  \textbf{8} (2008), no.~1, 45--80.

\bibitem{Bas05-top-errata}
\bysame, \emph{Errata for computing the top few {B}etti numbers of
  semi-algebraic sets defined by quadratic inequalities in polynomial time},
  Found. Comput. Math. \textbf{8} (2008), no.~1, 81--95.

\bibitem{Bas05-first-Kettner}
S.~Basu and M.~Kettner, \emph{A sharper estimate on the {B}etti numbers of sets
  defined by quadratic inequalities}, Discrete Comput. Geom. (to appear),
  preprint at arXiv:math.AG/0610954.

\bibitem{BPRbook2}
S.~Basu, R.~Pollack, and M.-F. Roy, \emph{Algorithms in real algebraic
  geometry}, Algorithms and Computation in Mathematics, vol.~10,
  Springer-Verlag, Berlin, 2006 (second edition). \MR{1998147 (2004g:14064)}

\bibitem{BPRbettione}
\bysame, \emph{Computing the first {B}etti number of a semi-algebraic set},
  Found. Comput. Math. \textbf{8} (2008), no.~1, 97--136.

\bibitem{Bas05-first}
Saugata Basu, \emph{Computing the first few {B}etti numbers of semi-algebraic
  sets in single exponential time}, J. Symbolic Comput. \textbf{41} (2006),
  no.~10, 1125--1154. \MR{2262087 (2007k:14120)}

\bibitem{Bas05-euler}
\bysame, \emph{Efficient algorithm for computing the {E}uler-{P}oincar\'e
  characteristic of a semi-algebraic set defined by few quadratic
  inequalities}, Comput. Complexity \textbf{15} (2006), no.~3, 236--251.
  \MR{2268404 (2007k:14119)}

\bibitem{BPR02}
Saugata Basu, Richard Pollack, and Marie-Fran{\c{c}}oise Roy, \emph{On the
  {B}etti numbers of sign conditions}, Proc. Amer. Math. Soc. \textbf{133}
  (2005), no.~4, 965--974 (electronic). \MR{2117195 (2006a:14096)}

\bibitem{BCR}
J.~Bochnak, M.~Coste, and M.-F. Roy, \emph{G\'eom\'etrie alg\'ebrique
  r\'eelle}, Ergebnisse der Mathematik und ihrer Grenzgebiete (3) [Results in
  Mathematics and Related Areas (3)], vol.~12, Springer-Verlag, Berlin, 1987.
  \MR{949442 (90b:14030)}

\bibitem{chern-spanier}
Shiing-shen Chern and E.~Spanier, \emph{The homology structure of sphere
  bundles}, Proc. Nat. Acad. Sci. U. S. A. \textbf{36} (1950), 248--255.
  \MR{0035994 (12,42h)}

\bibitem{GaV}
Andrei Gabrielov and Nicolai Vorobjov, \emph{Betti numbers of semialgebraic
  sets defined by quantifier-free formulae}, Discrete Comput. Geom. \textbf{33}
  (2005), no.~3, 395--401. \MR{2121987 (2005i:14075)}

\bibitem{GrPa2}
D.~Grigoriev and D.~Pasechnik, \emph{On {B}etti numbers of semi-algebraic sets
  over quadratic maps}, Manuscript, 2004.

\bibitem{GrPa04}
Dima Grigoriev and Dmitrii~V. Pasechnik, \emph{Polynomial-time computing over
  quadratic maps. {I}. {S}ampling in real algebraic sets}, Comput. Complexity
  \textbf{14} (2005), no.~1, 20--52. \MR{2134044 (2005m:68262)}

\bibitem{Hatcher}
Allen Hatcher, \emph{Algebraic topology}, Cambridge University Press,
  Cambridge, 2002. \MR{1867354 (2002k:55001)}

\bibitem{Milnor2}
J.~Milnor, \emph{On the {B}etti numbers of real varieties}, Proc. Amer. Math.
  Soc. \textbf{15} (1964), 275--280. \MR{0161339 (28 \#4547)}

\bibitem{OP}
I.~G. Petrovski{\u\i} and O.~A. Ole{\u\i}nik, \emph{On the topology of real
  algebraic surfaces}, Izvestiya Akad. Nauk SSSR. Ser. Mat. \textbf{13} (1949),
  389--402. \MR{0034600 (11,613h)}

\bibitem{Smale}
Stephen Smale, \emph{A {V}ietoris mapping theorem for homotopy}, Proc. Amer.
  Math. Soc. \textbf{8} (1957), 604--610. \MR{0087106 (19,302f)}

\bibitem{T}
Ren{\'e} Thom, \emph{Sur l'homologie des vari\'et\'es alg\'ebriques r\'eelles},
  Differential and Combinatorial Topology (A Symposium in Honor of Marston
  Morse), Princeton Univ. Press, Princeton, N.J., 1965, pp.~255--265.
  \MR{0200942 (34 \#828)}

\end{thebibliography}

\end{document}